\documentclass[11pt]{article}
\usepackage{amsthm}
\usepackage{amssymb}
\usepackage{enumerate}
\usepackage{multirow}
\usepackage{bbm}
\usepackage{amsfonts}
\usepackage{balance}
\usepackage{footnote}
\usepackage{graphicx}
\usepackage{tablefootnote}
\usepackage{xr}
\usepackage{hyperref}
\usepackage{amsmath}

\newcommand{\ex}[1]{\mathbb{E}\left[ #1 \right] }
\newcommand{\norm}[1]{\left\lVert #1 \right\rVert}
\newcommand{\norms}[1]{\lVert #1 \rVert}
\newcommand{\Q}{ {\mathbf Q}}

\newcommand{\A}{ {\mathbf A}}
\newcommand{\s}{ {\mathbf S}}
\newcommand{\U}{ {\mathbf U}}

\newcommand{\inner}[2]{\langle #1, #2 \rangle}
\newcommand{\lep}[1]{\mathop  \le \limits^{(#1)}}

\newcommand{\ep}[1]{\mathop  = \limits^{(#1)}}

\newtheorem{theorem}{Theorem}

\newtheorem{lemma}{Lemma}
\newtheorem{definition}{Definition}

\newtheorem{remark}{Remark}

\begin{document}
%
\title{A Note on Load Balancing in Many-Server Heavy-Traffic Regime }
\author{Xingyu Zhou\\Department of ECE\\The Ohio State University\\\texttt{zhou.2055@osu.edu}
\and Ness Shroff\\Department of ECE and CSE\\The Ohio State University\\
\texttt{shroff.11@osu.edu} } 
\date{}

\maketitle



\begin{abstract}
  In this note, we apply Stein's method to analyze the performance of general load balancing schemes in the many-server heavy-traffic regime. In particular, consider a load balancing system of $N$ servers and the distance of arrival rate to the capacity region is given by $N^{1-\alpha}$ with $\alpha > 1$. We are interested in the performance as $N$ goes to infinity under a large class of policies. We establish different asymptotics under different scalings and conditions. Specifically, (i) If the second moments linearly increase with $N$ with coefficients $\sigma_a^2$ and $\nu_s^2$, then for any $\alpha > 4$, the distribution of the sum queue length scaled by $N^{-\alpha}$ converges to an exponential random variable with mean $\frac{\sigma_a^2 + \nu_s^2}{2}$. (3) If the second moments quadratically increase with $N$ with coefficients $\tilde{\sigma}_a^2$ and $\tilde{\nu}_s^2$, then for any $\alpha > 3$, the distribution of the sum queue length scaled by $N^{-\alpha-1}$ converges to an exponential random variable with mean $\frac{\tilde{\sigma}_a^2 + \tilde{\nu}_s^2}{2}$. Both results are simple applications of our previously developed framework of Stein's method for heavy-traffic analysis in~\cite{zhou2020note}.
\end{abstract}

%


\section{Introduction}
Load balancing has attracted increasing attention recently due to its application in cloud computing and data centers. In this note, we consider a system consisting of one load balancer and $N$ servers each with an infinite buffer queue. The arrival is immediately dispatched to one of the servers based on a certain load balancing policy. In particular, we consider a set of systems where the distance of arrival rate to the capacity is given by $N^{1-\alpha}$ with $\alpha > 1$ and let $N$ go to infinity, which is often called the many-server heavy-traffic regime.

Many previous works have investigated the system performance under different values of $\alpha$. For example, if $\alpha = \frac{1}{2}$, i.e., Halfin-Whitt regime, Join-Shortest-Queue (JSQ) has been extensively studied~\cite{banerjee2019join,braverman2020steady,eschenfeldt2018join}. In the Sub-Halfin-Whitt regime where $\alpha \in (0,\frac{1}{2})$, several load balancing policies are investigated~\cite{liu2019simple}. Recently, the authors also extend the analysis to the case when $\alpha \in (\frac{1}{2}, 1)$~\cite{liu2019universal}. In~\cite{gupta2019load}, load balancing policies in Nondegenerate Slowdown regime (NDS) (i.e., $\alpha = 1$) are studied. 	
More recently, \cite{hurtado2020load} studied JSQ and they show that the total queue length scaled by $N^{-\alpha}$ converges to an exponential random variable via transform method and Stein's method. 

In this paper, instead of only focusing on JSQ policy under one particular scaling situation as in~\cite{hurtado2020load}, we investigate a large class of load balancing policies and establish their asymptotic performance for different values of $\alpha$ under different scalings. This is possible because we adopt the framework of Stein's method for heavy-traffic analysis developed in our early work~\cite{zhou2020note} for general load balancing and scheduling problems. In details, we have made the following key contributions.

First, we present the asymptotic performance for a large class of load balancing schemes. For any policy in this class, we show that the asymptotic performance depends on the scaling properties of the second moments of total arrival and service processes, i.e., $\sigma_{\Sigma}^{(N)}$ and $\nu_{\Sigma}^{(N)}$. In particular, if $\sigma_{\Sigma}^{(N)} = N \sigma_s^2$ and $\nu_{\Sigma}^{(N)} = N \nu_s^2$, then for any $\alpha > 4$, the distribution of the sum queue length scaled by $N^{-\alpha}$ converges to an exponential random variable with mean $\frac{\sigma_a^2 + \nu_s^2}{2}$. If $\sigma_{\Sigma}^{(N)} = N^2 \tilde{\sigma}_s^2$ and $\nu_{\Sigma}^{(N)} = N^2 \tilde{\nu}_s^2$, then for any $\alpha > 3$, the distribution of the sum queue length scaled by $N^{-\alpha-1}$ converges to an exponential random variable with mean $\frac{\tilde{\sigma}_a^2 + \tilde{\nu}_s^2}{2}$. 
It is worth noting that this class not only includes policies that achieve a single-dimensional state-space collapse (e.g., JSQ, Power-of-$d$, $p$-JSQ as in~\cite{zhou2018degree}, and many others in~\cite{zhou2017designing}), but also includes all the policies under which the state-space collapse region is multi-dimensional as long as it can be covered by a cone. On one hand, this directly indicates that a single-dimensional state-space collapse is not necessary for the asymptotic performance as in~\cite{hurtado2020load}. On the other hand, it also allows us to explore the trade-off between flexibility and performance. 

Second, although Stein's method serves as the key idea behind both~\cite{hurtado2020load} and our work, the execution in our work is totally different from~\cite{hurtado2020load}. In particular, our analysis is purely based on the general framework of Stein's method developed in our early work~\cite{zhou2020note}. This framework of Stein's method for heavy-traffic analysis can be used to analyze single-server system, general load balancing problems and scheduling problems. The result in this paper is just another application of our early framework with a very simple proof. By using this framework, we are not only able to establish asymptotic performance for a large class of policies, but also obtain different asymptotics under different scalings. The simplicity and broader applicability of our framework comes from the fact that it inherits the same intuitions and mathematical bounds as in the drift-based method. As a result, we can directly plug in previously well-known bounds established by drift-method into this framework, and hence easily establish new asymptotic performance beyond first moment result (e.g., convergence in distribution) without analyzing each policy by going through all the details repeatedly. For interesting readers, please refer to~\cite{zhou2020note} for more details.



\section{System model and preliminaries}
\label{sec:model}

We consider a single-hop queueing system in the discrete time, i.e., a time-slotted system. There are $N$ separate servers, each of them maintains an infinite capacity FIFO queue. Once a task or job is in a queue, it remains in that queue until its service is completed. Each server is
assumed to be work conserving, i.e., a server is idle if and only if its corresponding queue is empty.

Let $Q_n(t)$ be the queue length (i.e., tasks in the queue and the server) of server $n$ at the beginning of time-slot $t$. Let $A_{\Sigma}(t)$ denote denote the number of exogenous tasks that arrive at the beginning of time-slot $t$.  We assume that $A_{\Sigma}(t)$ is an integer-valued random variable with mean of $\lambda_{\Sigma}$, which is i.i.d. across time-slots. We further assume that there is a positive probability for $A_{\Sigma}(t)$ to be zero. We assume that $S_n(t)$
is also an integer-valued random variable with mean $\mu_n$, which is i.i.d. across time-slots. We also assume that $S_n(t)$ is independent across different servers as well as the arrival process. Let $S_{\Sigma}(t) \triangleq \sum_{n=1}^N S_n(t)$ denote the hypothetical total service process with mean of $\mu_{\Sigma} \triangleq \sum_{n=1}^N \mu_n$. We assume that both arrival and service processes have a bounded support, i.e., $A_{\Sigma}(t) \le A_{max}$ and $S_n(t) \le S_{max}$ for all $t$. 

We consider a set of load balancing systems parameterized by $\epsilon \triangleq N^{1-\alpha}$ such that $\lambda_{\Sigma}^{(\epsilon)} = \mu_{\Sigma} - \epsilon$ and $\mu_{\Sigma} = \theta(N)$, $A_{max} = \theta(N)$\footnote{This condition is necessary since the mean total arrival rate is on the order of $N$.}. In particular, we have $\lambda_{\Sigma}^{(\epsilon)} = \ex{\overline{A}_{\Sigma}}$, $(\sigma_{\Sigma}^{(\epsilon)})^2 = \text{Var}(\overline{A}_{\Sigma})$, $\mu_{\Sigma} = \ex{\overline{S}_{\Sigma}}$ and $\nu_{\Sigma}^2 = \text{Var}({\overline{S}_{\Sigma}})$. A load balancing policy is adopted by the dispatcher to determine to which queue the new arrivals should be sent.

In each time-slot, the order of events is as follows. First, queue lengths (or partial queue lengths) are observed. Based on these observations, a control problem is solved (i.e., the load balancing problem or the scheduling problem). Then, arrivals happen and the server processes tasks at the end of each time slot. In particular, the evolution of the length of queue $n$ is given by
\begin{align}
  Q_n(t+1) = Q_n(t) + A_n(t) - S_n(t) + U_n(t),
\end{align}
where $U_n(t) = \max(S_n(t)-A_n(t)-Q_n(t),0)$ is the unused service due to an empty queue.

In this paper, we add a line on top of variables and vectors to denote steady-state (e.g., $\overline{\Q}$, $\overline{\A}$ and $\overline{\s}$). In order to perform our heavy-traffic analysis, we consider a set of systems parametrized by a positive parameter $\epsilon$ (or equivalently by $N$). In particular, the parameter $\epsilon$ captures the distance of arrival vector to a particular point on the capacity region, i.e., a smaller $\epsilon$ means a heavier load. 

\begin{definition}
  A control policy is said to be throughput optimal if for any $\epsilon>0$, the system is positive recurrent and all the moments of $\norms{\overline{\Q}^{(\epsilon)}}$ are finite. 
\end{definition}

The main convergence metric used in this paper is the Wasserstein distance metric, which is defined as follows for non-negative random variables.
\begin{align*}
    d_W(X,Y) = \sup_{h\in\text{Lip}(1)} |\ex{h(X)} - \ex{h(Y)}|
  \end{align*}
where for a metric space $(\mathcal{S},d)$, $\text{Lip}(1)=\{h: \mathcal{S}\to\mathbb{R}, |h(x)-h(y)|\le d(x,y)\}$. The class $\text{Lip}(1)$ is simple to work with but at the same time rich enough so that convergence under the Wasserstein metric implies the convergence in distribution~\cite{gibbs2002choosing}.

\section{Main Results}
\label{sec:lb}
In this section, we directly apply the framework of Stein's method for heavy-traffic analysis developed in our early work~\cite{zhou2020note} to study load balancing in many-server heavy-traffic regime. As can be seen from the proof, all we need to do is basically replace $\epsilon$ by $N^{1-\alpha}$ and plug in previous bounds obtained via drift-based method. This directly implies the simplicity and general applicability of our framework.

\begin{lemma}
\label{lemma:1}
    Consider a set of load balancing systems parameterized by $N$ such that $\epsilon = N^{1-\alpha}$, $\alpha > 1$ with $\mu_{\Sigma} = \theta(N)$ and $A_{max} = \theta(N)$. Assume that $(\sigma_{\Sigma}^{(N)})^2 = N\sigma_a^2$ and $(\nu_{\Sigma}^{(N)})^2 = N\sigma_s^2$. Suppose that the load balancing policy is throughput optimal and there exists a function $g(N)$ such that 
    \begin{align}
    \label{eq:cross1}
     \frac{1}{N}\ex{\norms{\overline{\Q}^{(N)}(t+1)}_1\norms{\overline{\U}^{(N)}}_1} = O(g(N)).
    \end{align}
    Then, we have 
    \begin{align*}
       d_W(N^{-\alpha} \sum_{n=1}^N\overline{Q}_n^{(N)},Z) = O(\max(g(N),N^{2-\alpha})).
    \end{align*}
    where $Z \sim \text{Exp}(\frac{ 2}{\sigma_{a}^2 + \nu_{s}^2 })$. 
    
  \end{lemma}
  \begin{proof}
	The proof is a direct application of the framework of Stein's method developed in~\cite{zhou2020note}. The full proof is presented in Appendix~\ref{pf:lem1}.
  \end{proof}
    

  \begin{lemma}
\label{lemma:2}
    Consider a set of load balancing systems parameterized by $N$ such that $\epsilon = N^{1-\alpha}$, $\alpha > 1$  with $\mu_{\Sigma} = \theta(N)$ and $A_{max} = \theta(N)$. Assume that $({\sigma}_{\Sigma}^{(N)})^2 = N^2 \tilde{\sigma}_a^2$ and $(\nu_{\Sigma}^{(N)})^2 = N^2 \tilde{\sigma}_s^2$. Suppose that the load balancing policy is throughput optimal and there exists a function $g(N)$ such that 
    \begin{align}
    \label{eq:cross3}
      \frac{1}{N^2}\ex{\norms{\overline{\Q}^{(N)}(t+1)}_1\norms{\overline{\U}^{(N)}}_1} = O(g(N)).
    \end{align}
    Then, we have 
    \begin{align*}
       d_W(N^{-\alpha-1} \sum_{n=1}^N\overline{Q}_n^{(N)},Z) = O(\max(g(N),N^{-\alpha})).
    \end{align*}
    where $Z \sim \text{Exp}(\frac{ 2 }{\tilde{\sigma}_{a}^2 + \tilde{\nu}_{s}^2})$. 
  \end{lemma}
  \begin{proof}
  	The proof is nearly the same as that of Lemma~\ref{lemma:1}. See Appendix~\ref{pf:lem2}
  \end{proof}

  Now, armed with the lemmas above, we can directly analyze a class of load balancing schemes in the many-server heavy-traffic regime. In particular, we focus on the class introduced in one of our early works~\cite{zhou2018flexible}, which have been well-studied via drift-based method. Based on our framework, we can directly plug in the bounds obtained in the previous work to establish new asymptotic performance. 

  We first summarize the key ideas behind this class as follows. More details can be found in~\cite{zhou2018flexible}.

  Consider an $N$-dimensional cone $\mathcal{K}_{\gamma}$, which is finitely generated by a set of $N$ vectors $\{\mathbf{b}^{(n)}, n \in \mathcal{N}\}$, i.e., 
\begin{equation}
\label{eq:def_cone}
	\mathcal{K}_{\gamma} = \left\{ \mathbf{x} \in \mathbb{R}^N: \mathbf{x}  = \sum_{n \in \mathcal{N}} w_n\mathbf{b}^{(n)}, w_n \ge 0 \text{ for all } n \in \mathcal{N} \right\},
\end{equation}
where $\mathbf{b}^{(n)}$ is an $N$-dimensional vector with the $n$th component being $1$ and  $\gamma$ everywhere else for some $\gamma \in [0,1]$. It follows that, if $\gamma = 0$, the cone $\mathcal{K}_{\gamma}$ is the non-negative orthant of $\mathbb{R}^N$, and if $\gamma = 1$, the cone $\mathcal{K}_{\gamma}$ reduces to the single-dimensional line in which all the components are equal.

For a given cone $\mathcal{K}_{\gamma}$, we decompose $\overline{\Q}$ into two parts as follows
\begin{equation*}
	\overline{\Q} = \overline{\Q}_{\parallel} + \overline{\Q}_{\perp},
\end{equation*}
where $ \overline{\Q}_{\parallel}$ is the projection onto the cone $\mathcal{K}_{\gamma}$, referred to as the parallel component,  and $\overline{\Q}_{\perp}$ is the remainder, referred to as the perpendicular component

  Given a load balancing policy $\eta(t)$, we define the dispatching preference as 
  \begin{align*}
  	\Delta_{\eta(t)}(t) = \mathbf{P}_{\eta(t)}(t) - \mathbf{P}_{rand}(t),
  \end{align*}
  where $\mathbf{P}_{\eta(t)}(t)$ is the dispatching distribution vector and the $n$th component is the probability of selecting the $n$th shortest queue under $\eta(t)$. $\mathbf{P}_{rand}(t)$ is the dispatching distribution under (weighted) random routing.

  \begin{definition}[Flexible Class $\Pi_1$]
  	A load balancing scheme is said to be in the class $\Pi_1$ if there exists a cone $\mathcal{K}_{\gamma}$ such that for all $\Q(t) \notin \mathcal{K}_{\gamma}$,
  	\begin{enumerate}
  		\item there exists a $k \in \{2,3,\ldots, N\}$ such that ${\Delta}_n \ge 0$ for all $n<k$ and ${\Delta}_n \le 0$ for all $n\ge k$.
  		\item $\min(|{\Delta}_1|, |{\Delta}_N|) \ge \delta $ for some constant $\delta$.
  	\end{enumerate}
  \end{definition}
  \begin{remark}
  	The flexibility of this class comes from three dimensions: (a) it includes JSQ and Power-of-$d$ as special cases. Moreover, it also include many other useful policies as discussed in~\cite{zhou2017designing,zhou2018flexible}. (b) it does not require that the state-space collapse onto the line $\mathbf{c} = \{1,1,\ldots,1\}$ as in previous policies.(c) it also enables us to study the trade-off between flexibility and performance by scaling the constant $\delta$ and $\alpha$ with the load or the number of servers.
  \end{remark}
  \begin{theorem}
  \label{thm:1}
  	Given any load balancing scheme in class $\Pi_1$. Consider a set of load balancing systems parameterized by $N$ such that $\epsilon = N^{1-\alpha}$and $\mu_{\Sigma} = \theta(N)$, $A_{max} = \theta(N)$.
  	\begin{enumerate}
  		\item Assume that $(\sigma_{\Sigma}^{(N)})^2 = N\sigma_a^2$ and $(\nu_{\Sigma}^{(N)})^2 = N\sigma_s^2$. For any $r\ge 2$, we have
  	\begin{align*}
       d_W(N^{-\alpha} \sum_{n=1}^N\overline{Q}_n^{(N)},Z) = O(N^{4-\alpha + \frac{\alpha-1}{r}}).
    \end{align*}
    $Z \sim \text{Exp}(\frac{ 2}{\sigma_{a}^2 + \nu_{s}^2 })$. Thus, for any $\alpha>4$, the distribution of the sum queue length scaled by $N^{-\alpha}$ converges to an exponential random variable with mean $\frac{\sigma_{a}^2 + \nu_{s}^2}{2}$. 
    \footnote{Note that instead of $\alpha > 2$ in~\cite{hurtado2020load}, it indeed needs $\alpha > 4$ for the same result to hold since $A_{max}$ has to be $\theta(N)$ rather than a constant.}
    \item Assume that $({\sigma}_{\Sigma}^{(N)})^2 = N^2 \tilde{\sigma}_a^2$ and $(\nu_{\Sigma}^{(N)})^2 = N^2 \tilde{\sigma}_s^2$. We have for any $r\ge 2$
    \begin{align*}
       d_W(N^{-\alpha-1} \sum_{n=1}^N\overline{Q}_n^{(N)},Z) = O(N^{3-\alpha + \frac{\alpha-1}{r}}).
    \end{align*}
    $Z \sim \text{Exp}(\frac{ 2 }{\tilde{\sigma}_{a}^2 + \tilde{\nu}_{s}^2})$.  Thus, for any $\alpha > 3$, the distance approaches zero as $N \to \infty$ 
  	\end{enumerate}

  \end{theorem}

  \begin{proof}
  	Based on Lemmas~\ref{lemma:1} and~\ref{lemma:2}  all we need to study is the term $\ex{\norms{\overline{\Q}^{(N)}(t+1)}_1\norms{\overline{\U}^{(N)}}_1}$. In particular, it follows from the proof in~\cite{zhou2018flexible} that for any scheme in class $\Pi_1$ and any $r\ge 2$
  	\begin{align}
  	&\ex{\norms{\overline{\Q}^{(N)}(t+1)}_1\norms{\overline{\U}^{(N)}(t)}_1}\nonumber\\
  		\le &\frac{N}{\gamma}\ex{\inner{\overline{\U}}{-\overline{\Q}^+_{\perp}}}\nonumber\\
  		\le &\frac{N}{\gamma} { \left(\ex{\norm{\overline{\U}}^{r'}_{r'} } \right)^{\frac{1}{r'}}  \left(\ex{\norm{\overline{\Q}^+_{\perp}}^r_r} \right)^{\frac{1}{r}}}.\nonumber\\
    \le& \frac{N}{\gamma}  \left(c_{r'} \epsilon \right)^{\frac{1}{r'}} \left(\ex{\norm{\overline{\Q}^+_{\perp}}^r_2} \right)^{\frac{1}{r}}.\nonumber\\
    \le&\frac{N}{\gamma}  \left(c_{r'} \epsilon \right)^{\frac{1}{r'}} \left(\ex{\norm{\overline{\Q}_{\perp}}^r_2} \right)^{\frac{1}{r}}\nonumber\\
    \lep{a}& \frac{N}{\gamma \delta} (S_{max})^{\frac{1}{r}}K_r\epsilon^{1-1/r}\nonumber\\
    =&\frac{L_r}{\gamma\delta}N^{5-\alpha-\frac{1-\alpha}{r}},
  	\end{align}
  	where in (a) $K_r \triangleq \left[ \left(\frac{8NL}{\mu_{\Sigma}}\right)^r + r! \left(\frac{32D^2N + 4D\mu_{\Sigma}}{\mu_{\Sigma}}\right)^r \right]^{\frac{1}{r}}$, and $L = N\max(A_{max}, S_{max})^2$,  $D = \sqrt{N}\max(A_{max},S_{max})$. Since $A_{max} = \theta(N)$ and $\mu_{\Sigma} = \theta(N)$, we have $K_r \le L_r N^3$ for some constant $L_r$ independent of $N$. Thus, we have for any $r\ge2$
  	\begin{align*}
  		\ex{\norms{\overline{\Q}^{(N)}(t+1)}_1\norms{\overline{\U}^{(N)}(t)}_1} = O(N^{5-\alpha + \frac{\alpha-1}{r} }).
  	\end{align*}
  	Then, the results of Theorem~\ref{thm:1} directly follow from Lemmas~\ref{lemma:1} and~\ref{lemma:2}. 
  \end{proof}

\section{Conclusion}
In this note, we apply the recently developed framework of Stein's method for heavy-traffic analysis to study asymptotic performance of general load balancing schemes in many-server heavy-traffic regime. The main results can be easily obtained by plugging in well-known bounds obtained by drift-based method.

\bibliographystyle{plain}
\bibliography{ref} 

\section*{Appendix}
\appendix
\section{Proof of Lemma~\ref{lemma:1}}
\begin{proof}
\label{pf:lem1}
    Replace $\epsilon\norms{\overline{\Q}^{(\epsilon)}}_1$ in Eq.(6) of~\cite{zhou2020note} by $\hat{\epsilon}\norms{\overline{\Q}^{(\epsilon)}}_1$ with $\epsilon = N^{1-\alpha} = \mu_{\Sigma} - \lambda_{\Sigma}$ and $\hat{\epsilon} = N^{-\alpha}$. Taking expectation of both sides, yields
    \begin{align}
    \label{eq:stein_routing}
      \left|\ex{h(\hat{\epsilon} \norms{\overline{\Q}^{(\epsilon)}}_1)} - \ex{h(Z)}\right| = \left|\ex{\frac{1}{2}\sigma^2 f_h''\left(\hat{\epsilon}\norms{\overline{\Q}^{(\epsilon)}}_1\right)-\theta f_h'\left(\hat{\epsilon}\norms{\overline{\Q}^{(\epsilon)}}_1\right)}\right|
    \end{align}
    Now, we focus on the RHS. In particular, we have
    \begin{align*}
      &\ex{\frac{1}{2}\sigma^2 f_h''\left(\hat{\epsilon}\norms{\overline{\Q}^{(\epsilon)}}_1\right)-\theta f_h'\left(\hat{\epsilon}\norms{\overline{\Q}^{(\epsilon)}}_1\right)}\\
      \ep{a}&\ex{\frac{1}{2}\sigma^2 f_h''\left(\hat{\epsilon}\norms{\overline{\Q}^{(\epsilon)}}_1\right)-\theta f_h'\left(\hat{\epsilon}\norms{\overline{\Q}^{(\epsilon)}}_1\right) - \left(f_h\left(\hat{\epsilon}\norms{\overline{\Q}^{(\epsilon)}(t+1)}_1\right) -f_h\left(\hat{\epsilon}\norms{\overline{\Q}^{(\epsilon)}(t)}_1\right) \right)}\\
      =&\ex{\frac{1}{2}\sigma^2 f_h''\left(\hat{\epsilon}\norms{\overline{\Q}}_1\right)-\theta f_h'\left(\hat{\epsilon}\norms{\overline{\Q}}_1\right)} - \ex{f_h\left(\hat{\epsilon}(\norms{\overline{\Q}(t)}_1 + \norms{\overline{\A}(t)}_1 - \norms{\overline{\s}(t)}_1 + \norms{\overline{\U}(t)}_1)\right) - f_h\left(\hat{\epsilon}\norms{\overline{\Q}}_1\right)}
    \end{align*}
    where (a) holds since the policy is throughput optimal and the result (a) in Lemma 1 of~\cite{zhou2020note}.

    For the second expectation, we have
    \begin{align*}
      &\ex{f_h\left(\hat{\epsilon}(\norms{\overline{\Q}(t)}_1 + \norms{\overline{\A}(t)}_1 - \norms{\overline{\s}(t)}_1 + \norms{\overline{\U}(t)}_1)\right) - f_h\left(\hat{\epsilon}\norms{\overline{\Q}}_1\right)}\\
      =&{\ex{\hat{\epsilon}^2\frac{f_h''(\hat{\epsilon}\norms{\overline{\Q}}_1)}{2}\left(\norms{\overline{\A}}_1-\norms{\overline{\s}}_1 \right)^2 + \hat{\epsilon} f_h'(\hat{\epsilon} \norms{\overline{\Q}}_1)\left(\norms{\overline{\A}}_1-\norms{\overline{\s}}_1\right)}} \\
      &+\ex{\hat{\epsilon}^3\frac{f_h'''(\eta)}{6} \left(\norms{\overline{\A}}_1-\norms{\overline{\s}}_1\right)^3 + \hat{\epsilon}\norms{\overline{\U}}_1f_h'(\hat{\epsilon}\norms{\overline{\Q}(t+1)}_1) -\hat{\epsilon}^2\frac{f_h''(\xi)}{2}\norms{\overline{\U}}_1^2 }\\
      =&\ex{\hat{\epsilon}^2\frac{f_h''(\hat{\epsilon}\norms{\overline{\Q}}_1)}{2}\left( N\sigma_a^2 + N\nu_s^2\right) -N\hat{\epsilon}^2f_h'(\hat{\epsilon}\norms{\overline{\Q}}_1)}\\
      &+\ex{\hat{\epsilon}^4\frac{f_h''(\hat{\epsilon}\norms{\overline{\Q}}_1)}{2} + \hat{\epsilon}^3\frac{f_h'''(\eta)}{6} \left(\norms{\overline{\A}}_1-\norms{\overline{\s}}_1\right)^3 + \hat{\epsilon}\norms{\overline{\U}}_1f_h'(\hat{\epsilon}\norms{\overline{\Q}(t+1)}_1) -\hat{\epsilon}^2\frac{f_h''(\xi)}{2}\norms{\overline{\U}}_1^2 }
    \end{align*}
    Now, let $\sigma^2 =N\hat{\epsilon}^2\left(\sigma_a^2 + \nu_s^2\right)$ and $\theta = N\hat{\epsilon}^2$ in Eq.~\eqref{eq:stein_routing}, we have 
    \begin{align*}
      \left|\ex{h(\hat{\epsilon}\norms{\overline{\Q}^{(\epsilon)}}_1)} - \ex{h(Z)}\right| &\le \underbrace{\ex{\left|\hat{\epsilon}^3\frac{f_h'''(\eta)}{6} \left(\norms{\overline{\A}}_1-\norms{\overline{\s}}_1\right)^3\right| + \left|\hat{\epsilon}^2\frac{f_h''(\xi)}{2}\norms{\overline{\U}}_1^2\right| + \left|\hat{\epsilon}^4\frac{f_h''(\hat{\epsilon}\norms{\overline{\Q}}_1)}{2}\right|  }}_{\mathcal{T}_1}\\
      &+ \underbrace{\ex{\left|\hat{\epsilon}\norms{\overline{\U}}_1f_h'(\hat{\epsilon}\norms{\overline{\Q}(t+1)}_1)\right|}}_{\mathcal{T}_2}\nonumber
    \end{align*}

    For $\mathcal{T}_1$, we have 
    \begin{align*}
      \mathcal{T}_1 &\le \hat{\epsilon}^3\frac{\norm{f_h'''}}{6}\ex{\overline{A}_{\Sigma}^3 + \overline{S}_{\Sigma}^3 + 3\mu_{\Sigma}(\overline{A}_{\Sigma}^2 + \overline{S}_{\Sigma}^2)} + \hat{\epsilon}^2\frac{\norm{f_h''}}{2}\ex{\norms{\overline{\U}}_1^2} + \hat{\epsilon}^4 \frac{\norm{f_h''}}{2}\\
      &\le \frac{2\hat{\epsilon} }{3\left(N\sigma_a^2 + N\nu_s^2\right)}\ex{\overline{A}_{\Sigma}^3 + \overline{S}_{\Sigma}^3 + 3\mu_{\Sigma}(\overline{A}_{\Sigma}^2 + \overline{S}_{\Sigma}^2)} + \frac{1}{2N}\ex{\norms{\overline{\U}}_1^2} + \frac{1}{2N}\hat{\epsilon}^2\\
      &\lep{a} O(N^2\hat{\epsilon} ) + S_{max}\ex{\norms{\overline{\U}}_1}\\
      &\ep{b} O(N^{2-\alpha}),
    \end{align*}
    where (a) holds since $A_{max} = \theta(N)$, $\mu_{\Sigma} = \theta(N)$ and $S_{max}$ is a constant independent of $N$; (b) is true since $\ex{\norms{\overline{\U}}_1} =\epsilon = N^{1-\alpha}$ and $\hat{\epsilon} = N^{-\alpha}$.

    For $\mathcal{T}_2$, we have
    \begin{align*}
      \mathcal{T}_2 &= \ex{\left|\hat{\epsilon}\norms{\overline{\U}}_1f_h'(\hat{\epsilon}\norms{\overline{\Q}(t+1)}_1) - \hat{\epsilon}\norms{\overline{\U}}_1f_h'(0)\right|}\\
      &=\ex{\left|\hat{\epsilon}^2\norms{\overline{\Q}(t+1)}_1\norms{\overline{\U}}_1f_h''(\zeta)\right|}\\
      &\le \frac{1}{N}\ex{\norms{\overline{\Q}(t+1)}_1\norms{\overline{\U}}_1}\\
      & = O(g(N))
    \end{align*}
    Thus, we have 
    \begin{align*}
       \left|\ex{h(\epsilon\norms{\overline{\Q}^{(\epsilon)}}_1)} - \ex{h(Z)}\right| \le \mathcal{T}_1 + \mathcal{T}_2 = O(\max(g(N), N^{2-\alpha} )),
    \end{align*}
    which completes the proof of Lemma~\ref{lemma:1}.
 \end{proof}

\section{Proof of Lemma~\ref{lemma:2}}
\begin{proof}
\label{pf:lem2}
It follows exactly the same procedure as the proof of Lemma~\ref{lemma:1} with $\hat{\epsilon} = N^{-\alpha - 1}$, $\sigma^2 =N^2\hat{\epsilon}^2\left(\tilde{\sigma}_a^2 + \tilde{\nu}_s^2\right)$ and $\theta = N^2\hat{\epsilon}^2$. For $\mathcal{T}_1$, we have 

For $\mathcal{T}_1$, we have 
    \begin{align*}
      \mathcal{T}_1 &\le \hat{\epsilon}^3\frac{\norm{f_h'''}}{6}\ex{\overline{A}_{\Sigma}^3 + \overline{S}_{\Sigma}^3 + 3\mu_{\Sigma}(\overline{A}_{\Sigma}^2 + \overline{S}_{\Sigma}^2)} + \hat{\epsilon}^2\frac{\norm{f_h''}}{2}\ex{\norms{\overline{\U}}_1^2} + \hat{\epsilon}^4 \frac{\norm{f_h''}}{2}\\
      &\le \frac{2\hat{\epsilon} }{3\left(N^2 \tilde{\sigma}_a^2 + N^2\tilde{\nu}_s^2\right)}\ex{\overline{A}_{\Sigma}^3 + \overline{S}_{\Sigma}^3 + 3\mu_{\Sigma}(\overline{A}_{\Sigma}^2 + \overline{S}_{\Sigma}^2)} + \frac{1}{2N^2}\ex{\norms{\overline{\U}}_1^2} + \frac{1}{2N^2}\hat{\epsilon}^2\\
      &\lep{a} O(N\hat{\epsilon} ) + \frac{1}{N}S_{max}\ex{\norms{\overline{\U}}_1}\\
      &\ep{b} O(N^{-\alpha}),
    \end{align*}
    where (a) holds since $A_{max} = \theta(N)$, $\mu_{\Sigma} = \theta(N)$ and $S_{max}$ is a constant independent of $N$; (b) is true since $\ex{\norms{\overline{\U}}_1} =\epsilon = N^{1-\alpha}$ and $\hat{\epsilon} = N^{-\alpha}$.

For $\mathcal{T}_2$, we have 
\begin{align*}
	\mathcal{T}_2 \le \frac{1}{N^2}\ex{\norms{\overline{\Q}(t+1)}_1\norms{\overline{\U}}_1} = O(g(N)).
\end{align*}
\end{proof}

\end{document}